\documentclass{amsart}
\usepackage{amssymb,amstext,amsmath,amscd,amsthm,amsfonts,enumerate,graphicx,latexsym,stmaryrd,multicol}

\usepackage[usenames]{color}
\usepackage[all]{xy}
\newtheorem{thm}{Theorem}[section]
\newtheorem{lem}[thm]{Lemma}
\newtheorem{prop}[thm]{Proposition}
\newtheorem{cor}[thm]{Corollary}
\theoremstyle{definition}

\newtheorem{ques}[thm]{Question}
\newtheorem{rem}[thm]{Remark}
\newtheorem{conj}[thm]{Conjecture}

\newtheorem{rmk}[thm]{Remark}
\theoremstyle{remark}

\newtheorem*{ac}{Acknowlegments}
\newtheorem*{conv}{Convention}

\numberwithin{equation}{thm}
\def\ass{\operatorname{Ass}}
\def\cm{\mathsf{CM}}
\def\cok{\operatorname{Coker}}
\def\depth{\operatorname{depth}}
\def\Ext{\operatorname{Ext}}

\def\G{\mathcal{G}}
\def\ge{\geqslant}
\def\height{\operatorname{ht}}
\def\Hom{\operatorname{Hom}}
\def\image{\operatorname{Im}}

\def\le{\leqslant}

\def\m{\mathfrak{m}}
\def\mod{\operatorname{mod}}
\def\N{\mathbb{N}}
\def\n{\mathfrak{n}}
\def\nf{\mathrm{NF}}
\def\P{\mathbb{P}}
\def\p{\mathfrak{p}}
\def\pd{\operatorname{pd}}
\def\Q{\mathbb{Q}}
\def\q{\mathfrak{q}}
\def\s{\mathrm{S}}
\def\spec{\operatorname{Spec}}
\def\supp{\operatorname{Supp}}
\def\syz{\Omega}
\def\t{\mathrm{t}}
\def\Tor{\operatorname{Tor}}
\def\tr{\operatorname{Tr}}
\def\V{\mathrm{V}}
\def\x{\operatorname{X}}
\def\xx{\boldsymbol{x}}
\begin{document}
\allowdisplaybreaks
\title[MCM tensor products and vanishing of Ext]{Maximal Cohen--Macaulay tensor products\\
and vanishing of Ext modules}
\author{Kaito Kimura}
\address[KK]{Graduate School of Mathematics, Nagoya University, Furocho, Chikusaku, Nagoya 464-8602, Japan}
\email{m21018b@math.nagoya-u.ac.jp}
\author{Yuya Otake}
\address[YO]{Graduate School of Mathematics, Nagoya University, Furocho, Chikusaku, Nagoya 464-8602, Japan}
\email{m21012v@math.nagoya-u.ac.jp}
\author{Ryo Takahashi}
\address[RT]{Graduate School of Mathematics, Nagoya University, Furocho, Chikusaku, Nagoya 464-8602, Japan}
\email{takahashi@math.nagoya-u.ac.jp}
\urladdr{https://www.math.nagoya-u.ac.jp/~takahashi/}
\thanks{RT was partly supported by JSPS Grant-in-Aid for Scientific Research 19K03443}
\begin{abstract}
In this paper, we investigate the maximal Cohen--Macaulay property of tensor products of modules, and then give criteria for projectivity of modules in terms of vanishing of Ext modules.
One of the applications shows that the Auslander--Reiten conjecture holds for Cohen--Macaulay normal rings.
\end{abstract}
\maketitle
\section{Introduction}

Torsion in tensor products of (finitely generated) modules over a (commutative noetherian) local ring is one of the most classical subjects in commutative algebra.
It has been studied by Auslander \cite{A} over a regular local ring, and by Huneke and Wiegand \cite{HW} over a local hypersurface.
Recently, Celikbas and Sadeghi \cite{CS} have found a necessary condition for tensor products of modules to be maximal Cohen--Macaulay over a certain local complete intersection.

Vanishing of Ext modules over a local ring is also an actively studied subject in commutative algebra.
A lot of criteria for a given module to be projective have been described in terms of vanishing of Ext modules so far; see \cite{Ar,secm,ADS,ABS,DEL,HL,LM,ST} for instance.
One of the most important problems is a celebrated long-standing conjecture called the {\em Auslander--Reiten conjecture} \cite{AR}.

The purpose of this paper is to proceed with the study of the above two subjects.
We first explore the maximal Cohen--Macaulay property of tensor products of modules.
The main result in this direction is the following theorem.
For each module $E$ over a ring $R$, set $(-)^E=\Hom_R(-,E)$.

\begin{thm}[Theorem \ref{4}]
Let $R$ be a Cohen--Macaulay local ring with a canonical module $\omega$.
Let $M$ be a finitely generated $R$-module, and let $N$ be a maximal Cohen--Macaulay $R$-module.
Then $M\otimes_RN$ is a maximal Cohen--Macaulay $R$-module if and only if so is $\tr M\otimes_RN^\omega$.
\end{thm}

As an application of the above theorem, we recover in Corollary \ref{lem1} a result in \cite{LM} on the relationship between maximal Cohen--Macaulayness of tensor products and vanishing of Ext modules.
The original proof given in \cite{LM} is based on the techniques of Hilbert--Samuel multiplicity.
Our proof is simple and elementary without using multiplicity theory; only fundamental properties of a canonical module suffice.

Next we consider vanishing of Ext modules over a Cohen--Macaulay local ring, using our observations on maximal Cohen--Macaulay tensor products.
Our main result in this direction is the following theorem, which is a special case of Theorems \ref{them1} and \ref{k1}.

\begin{thm}[Theorems \ref{them1} and \ref{k1}]
Let $R$ be a Cohen--Macaulay local ring of dimension $d\ge2$.
Let $M$ be a finitely generated $R$-module which is locally of finite projective dimension in codimension one.
\begin{enumerate}[\rm(1)]
\item
Let $C$ be a maximal Cohen--Macaulay $R$-module with full support.
Suppose that $\Ext_R^i(M,C)=0$ for all $1\le i\le d$ and $\Ext_R^j(M^R,M^C)=0$ for all $1\le j\le d-1$.
Then $M$ is free.
\item
If $\Ext^i_R(M,R)=0$ for all $1 \le i \le 2d+1$ and $\Ext^j_R(M,M)=0$ for all $1 \le j \le d-1$, then $M$ is free.
\end{enumerate}
\end{thm}

The above theorem (more precisely, Theorems \ref{them1} and \ref{k1}) yields various applications.
Among other things, the following corollary is obtained.

\begin{cor}[Corollary \ref{k2}]
Let $R$ be a Cohen--Macaulay ring.
If $R_\p$ satisfies the Auslander--Reiten conjecture for every $\p\in\spec R$ with $\height\p\le1$, then $R_\p$ satisfies that conjecture for every $\p\in\spec R$.
In particular, the Auslander--Reiten conjecture holds true for an arbitrary Cohen--Macaulay normal ring.
\end{cor}

The Auslander--Reiten conjecture asserts that every module $M$ over a ring $R$ such that $\Ext_R^{>0}(M,M\oplus R)=0$ is projective.
This conjecture is known to hold true if $R$ is a locally excellent Cohen--Macaulay normal ring containing the field of rational numbers \cite{HL}, or if $R$ is a Gorenstein normal ring \cite{Ar}, or if $R$ is a local Cohen--Macaulay normal ring and $M$ is a maximal Cohen--Macaulay module such that $\Hom_R(M,M)$ is free \cite{DEL}.
The above corollary gives a common generalization of these three facts.
Other applications of our theorem recover and refine a lot of results in the literature; see Remark \ref{12} for details.

We close the section by stating our convention adopted throughout the remainder of this paper.

\begin{conv}
Let $R$ be a commutative noetherian ring.
All modules are assumed to be finitely generated.
Call an $R$-module $M$ {\em maximal Cohen--Macaulay} if $\depth M_\p=\dim R_\p$ for all $\p\in\supp M$, so that the zero module is thought of as maximal Cohen--Macaulay.
Denote by $\mod R$ the category of (finitely generated) $R$-modules.
Set $(-)^\ast=\Hom_R(-,R)$, and let $(-)^\vee$ be the Matlis dual.
Whenever $R$ is a Cohen--Macaulay local ring with a canonical module $\omega$, put $(-)^\dag=\Hom_R(-,\omega)$.
The set of nonnegative integers is denoted by $\N$, while the transpose of a matrix $A$ by ${}^\t\! A$.
We omit subscripts/superscripts if there is no ambiguity.
\end{conv}

\section{Maximal Cohen--Macaulay tensor products}

The main theme of this section is maximal Cohen--Macaulayness of tensor products of modules.
We give a certain operation which preserves this property, and then state applications.
We begin with recalling a couple of elementary facts on maximal Cohen--Macaulay modules.

\begin{rem}\label{3}
Let $R$ be a Cohen--Macaulay local ring.
Denote by $\cm(R)$ the full subcategory of $\mod R$ consisting of maximal Cohen--Macaulay $R$-modules.
Let $0\to L\to M\to N\to0$ be an exact sequence of $R$-modules.
If $M$ and $N$ belong to $\cm(R)$, then so does $L$.
If $R$ admits a canonical module, then the assignment $X\mapsto X^\dag$ gives a duality of the additive category $\cm(R)$.
\end{rem}

Let $R$ be a local ring and $M$ an $R$-module.
Denote by $\syz M$ and $\tr M$ the {\em (first) syzygy} and {\em (Auslander) transpose} of $M$, respectively.
These are defined as follows.
Let $F=(\cdots\xrightarrow{\partial_3}F_2\xrightarrow{\partial_2}F_1\xrightarrow{\partial_1}F_0\to0)$ be a minimal free resolution of $M$.
Then $\syz M=\image\partial_1$ and $\tr M=\cok(\partial_1^\ast)$.
We put $\syz^0M=M$, and for an integer $n>0$ the {\em $n$th syzygy} $\syz^nM$ is defined inductively by $\syz^nM=\syz(\syz^{n-1}M)$, whence $\syz^nM\cong\image\partial_n$.
The modules $\tr M$ and $\syz^nM$ for each $n\ge0$ are uniquely determined up to isomorphism since so is a minimal free resolution $F$ of $M$.
Note that $M^\ast\cong\syz^2\tr M$ up to free summands.

The main result of this section is the following theorem.

\begin{thm}\label{4}
Let $R$ be a Cohen--Macaulay local ring with a canonical module.
Let $M,N$ be $R$-modules and assume that $N$ is maximal Cohen--Macaulay.
Then $M\otimes_RN$ is a maximal Cohen--Macaulay $R$-module if and only if so is $\tr M\otimes_RN^\dag$.
\end{thm}

\begin{proof}
We freely use Remark \ref{3}.
Take an exact sequence $\alpha:R^{\oplus n}\xrightarrow{A}R^{\oplus m}\to M\to0$, where $A$ is an $m\times n$ matrix.
Dualizing $\alpha$ by $R$ gives an exact sequence $R^{\oplus m}\xrightarrow{{}^\t\! A}R^{\oplus n}\to\tr M\to0$.
Tensoring $N^\dag$ with it, we get an exact sequence $\beta:(N^\dag)^{\oplus m}\xrightarrow{{}^\t\! A}(N^\dag)^{\oplus n}\to\tr M\otimes N^\dag\to0$.
Dualizing $\alpha$ by $N^\dag$ gives an exact sequence $0\to\Hom(M,N^\dag)\to(N^\dag)^{\oplus m}\xrightarrow{{}^\t\! A}(N^\dag)^{\oplus n}$.
Splicing this with $\beta$ yields an exact sequence
\begin{equation}\label{1}
0\to\Hom(M,N^\dag)\to(N^\dag)^{\oplus m}\xrightarrow{{}^\t\! A}(N^\dag)^{\oplus n}\to\tr M\otimes N^\dag\to0.
\end{equation}
Tensoring $N$ with $\alpha$ gives an exact sequence $\gamma:N^{\oplus n}\xrightarrow{A}N^{\oplus m}\to M\otimes N\to0$.
Applying $(-)^\dag$ to $\beta$, we get an exact sequence $0\to(\tr M\otimes N^\dag)^\dag\to N^{\oplus n}\xrightarrow{A}N^{\oplus m}$.
Splicing this with $\gamma$ yields an exact sequence
\begin{equation}\label{2}
0\to(\tr M\otimes N^\dag)^\dag\to N^{\oplus n}\xrightarrow{A}N^{\oplus m}\to M\otimes N\to0.
\end{equation}

Assume that $M\otimes_RN$ is maximal Cohen--Macaulay.
Then all the terms of \eqref{2} are maximal Cohen--Macaulay.
Applying $(-)^\dag$ to \eqref{2} induces an exact sequence
$$
0\to(M\otimes N)^\dag\to(N^\dag)^{\oplus m}\xrightarrow{{}^\t\! A}(N^\dag)^{\oplus n}\to(\tr M\otimes N^\dag)^{\dag\dag}\to0,
$$
all of whose terms are maximal Cohen--Macaulay.
It follows from this and \eqref{1} that $\tr M\otimes N^\dag\cong\cok((N^\dag)^{\oplus m}\xrightarrow{{}^\t\! A}(N^\dag)^{\oplus n})\cong(\tr M\otimes N^\dag)^{\dag\dag}$, and thus $\tr M\otimes N^\dag$ is maximal Cohen--Macaulay.

Conversely, assume $\tr M\otimes N^\dag$ is a maximal Cohen--Macaulay $R$-module.
Then all the terms of \eqref{1} are maximal Cohen--Macaulay.
Applying $(-)^\dag$ to \eqref{1}, we get an exact sequence
$$
0\to(\tr M\otimes N^\dag)^\dag\to N^{\oplus n}\xrightarrow{A}N^{\oplus m}\to\Hom(M,N^\dag)^\dag\to0,
$$
and all the terms are maximal Cohen--Macaulay.
This sequence together with \eqref{2} yields $M\otimes N\cong\cok(N^{\oplus n}\xrightarrow{A}N^{\oplus m})\cong\Hom(M,N^\dag)^\dag$, and therefore $M\otimes N$ is a maximal Cohen--Macaulay module.
\end{proof}

We note an immediate corollary of Theorem \ref{4}.

\begin{cor}
Let $R$ be a Cohen--Macaulay local ring with a canonical module $\omega$, and let $M$ be an $R$-module.
Then, $M$ is maximal Cohen--Macaulay if and only if so is $\tr M\otimes_R\omega$.
The module $\tr M$ is maximal Cohen--Macaulay if and only if so is $M\otimes_R\omega$.
\end{cor}

\begin{proof}
Applying Theorem \ref{4} to $N=R$ (resp. $N=\omega$) implies the first (resp. second) assertion.
\end{proof}

We can deduce a result of Lyle and Monta\~no \cite[Lemma 3.4(1)]{LM} directly from Theorem \ref{4}.
The proof of \cite[Lemma 3.4(1)]{LM} refers to that of \cite[Lemma 5.3]{DEL}, where the same techniques of (Hilbert--Samuel) multiplicity as in the proof of \cite[Lemma 1.6]{HL} play an essential role.
It is worth mentioning that our methods are irrelevant to multiplicity theory, and are indeed simpler and more elementary.

\begin{cor}[Lyle--Monta\~no]\label{lem1}
Let $R$ be a $d$-dimensional Cohen--Macaulay local ring with a canonical module.
Let $M,N$ be $R$-modules and assume that $N$ is maximal Cohen--Macaulay.
If $\Ext_R^i(M,N)=0$ for all $1\le i\le d$, then the tensor product $M\otimes_R N^\dagger$ is a maximal Cohen--Macaulay $R$-module.
\end{cor}

\begin{proof}
Let $F$ be a minimal free resolution of $M$.
There are exact sequences $F_0^\ast\otimes N\to F_1^\ast\otimes N\to\tr M\otimes N\to0$ and $0\to\Hom(M,N)\to\Hom(F_0,N)\to\Hom(F_1,N)$.
Since $F_i^\ast\otimes N\cong\Hom(F_i,N)$ for $i=0,1$, we obtain an exact sequence
$$
\alpha:0\to\Hom(M,N)\to\Hom(F_0,N)\to\Hom(F_1,N)\to\tr M\otimes N\to0.
$$
The assumption that $\Ext_R^i(M,N)=0$ for all $1\le i\le d$ implies that the induced sequence
$$
\beta:0\to\Hom(M,N)\to\Hom(F_0,N)\to\Hom(F_1,N)\to\Hom(F_2,N)\to\cdots\to\Hom(F_{d+1},N)
$$
is exact.
The exact sequences $\alpha,\beta$ induce an exact sequence
$$
0\to\tr M\otimes N\to\Hom(F_2,N)\to\Hom(F_3,N)\to\cdots\to\Hom(F_{d+1},N).
$$
As $\Hom(F_i,N)$ is maximal Cohen--Macaulay for all $2\le i\le d+1$, so is $\tr M\otimes N$ by the depth lemma.
By virtue of Theorem \ref{4} (and Remark \ref{3}), the $R$-module $M\otimes_R N^\dagger$ is maximal Cohen--Macaulay.
\end{proof}

As stated above, the assertion of Corollary \ref{lem1} is the same as \cite[Lemma 3.4(1)]{LM}, and it extends the former half of \cite[Lemma 5.3]{DEL}.
The latter half of \cite[Lemma 5.3]{DEL} asserts that for a Cohen--Macaulay local ring $R$ with a canonical module $\omega$ and maximal Cohen--Macaulay $R$-modules $M,N$, if $\Ext_R^i(M,N)$ have finite length for all $1\le i\le d$ and $M\otimes_RN^\dag$ is maximal Cohen--Macaulay, then $\Ext_R^i(M,N)=0$ for all $1\le i\le d$.
In relation to this, we can show the following proposition.
Note that the proof is direct and independent of Theorem \ref{4}.

\begin{prop}\label{5}
Let $R$ be a $d$-dimensional Cohen--Macaulay local ring with a canonical module $\omega$.
Let $M$ and $N$ be $R$-modules.
Assume that $N$ is maximal Cohen--Macaulay, and $\dim_R\Tor_i^R(M,N)<i$ for all $1\le i\le d$.
Then $M\otimes_RN$ is maximal Cohen--Macaulay if and only if $\Ext_R^i(M,N^\dag)=0$ for all $1\le i\le d$.
\end{prop}

\begin{proof}
As $N$ is maximal Cohen--Macaulay, there exists a spectral sequence $E_2^{p,q}=\Ext^p(\Tor_q(M,N),\omega)\Rightarrow H^{p+q}=\Ext^{p+q}(M,N^\dag)$.
If $1\le q\le d$ and $p+q\le d$, then $d-\dim\Tor_q(M,N)>d-q\ge p$, and $E_2^{p,q}=0$ by \cite[Corollary 3.5.11(a)]{BH}.
Hence $H^i\cong E_2^{i,0}$ for each $0\le i\le d$.
Thus, $\Ext^i(M,N^\dag)=0$ for all $1\le i\le d$ if and only if $\Ext^i(M\otimes N,\omega)=0$ for all $1\le i\le d$, if and only if $M\otimes N$ is maximal Cohen--Macaulay.
\end{proof}

We record a consequence of the combination of Proposition \ref{5} and Theorem \ref{4}.

\begin{cor}\label{a}
Let $(R,\m)$ be a $d$-dimensional Cohen--Macaulay local ring with a canonical module, and let $M,N$ be $R$-modules.
Suppose that $N$ is a maximal Cohen--Macaulay module and $\dim_R\Tor_i^R(M,N^\dag)<i$ for all $1\le i\le d$.
Then $\Ext_R^i(M,N)=0$ for all $1\le i\le d$ if and only if the $R$-module $\tr M\otimes_RN$ is maximal Cohen--Macaulay.
\end{cor}

Finally, we compare our results with one in \cite{ST}.
For an $R$-module $M$ we denote by $\nf(M)$ the {\em nonfree locus} of $M$, that is, the set of prime ideals $\p$ of $R$ such that the localization $M_\p$ is nonfree over $R_\p$.

\begin{rem}
Let $(R,\m)$ be a $d$-dimensional Cohen--Macaulay local ring with a canonical module.
Let $M$ and $N$ be $R$-modules.
Suppose that $N$ is maximal Cohen--Macaulay, that $M$ is locally totally reflexive on the punctured spectrum of $R$, and that $\nf(M)\cap\nf(N)\subseteq\{\m\}$.
Then it holds that $\Ext_R^i(M,N)^\vee\cong\Ext_R^{(d+1)-i}(\tr M,N^\dag)$ for all $1\le i\le d$ by \cite[Theorem 2.8]{ST}.
We thus obtain the top equivalence in the following diagram, while the other three equivalences follow from our results stated above.
$$
\xymatrix@C+2pc{
\Ext_R^i(M,N)=0\text{ for all }1\le i\le d\ \ar@{<=>}[r]^-{\text{\cite[Theorem 2.8]{ST}}}\ar@{<=>}[d]^{\text{Proposition \ref{5}}}& \ \Ext_R^i(\tr M,N^\dag)=0\text{ for all }1\le i\le d\ar@{<=>}[d]^{\text{Proposition \ref{5}}}\\
M\otimes_RN^\dag\text{ is maximal Cohen--Macaulay\phantom{.}}\ \ar@{<=>}[r]^-{\text{Theorem \ref{4}}}& \ \tr M\otimes_RN\text{ is maximal Cohen--Macaulay.}
}
$$
\end{rem}

\section{Vanishing of Ext modules}

In this section, using what is given in the previous section, we provide various criteria for modules to be projective in terms of vanishing of Ext modules, and recover/refine a lot of results in the literature.

We begin with proving our key proposition, and for this we need a lemma.
For $R$-modules $X$ and $Y$, we define homomorphisms
$$
\alpha_{X,Y} : X\to \Hom_R(Y,Y\otimes_R X),\qquad
\beta_{X,Y} : X\otimes_R \Hom_R(X,Y)\to Y
$$
by $\alpha_{X,Y}(x)(y)=y\otimes x$ and $\beta_{X,Y}(x\otimes f)=f(x)$ for $x\in X$, $y\in Y$ and $f\in \Hom(X,Y)$.

\begin{lem}\label{6}
Let $X$ and $Y$ be $R$-modules.
Then $\beta_{Y,Y\otimes X}$ is surjective and $\alpha_{\Hom(X,Y),X}$ is injective.
\end{lem}

\begin{proof}
It is straightforward to verify that the compositions
$$
\begin{array}{rl}
&Y\otimes X\xrightarrow{Y\otimes\alpha_{X,Y}} Y\otimes\Hom(Y,Y\otimes X) \xrightarrow{\beta_{Y,Y\otimes X}} Y\otimes X,\\
&\Hom(X,Y)\xrightarrow{\alpha_{\Hom(X,Y),X}}\Hom(X,X\otimes\Hom(X,Y))\xrightarrow{\Hom(X,\beta_{X,Y})}\Hom(X,Y)
\end{array}
$$
are identity maps.
Hence $\beta_{Y,Y\otimes X}$ is a split epimorphism and $\alpha_{\Hom(X,Y),X}$ is a split monomorphism.
\end{proof}

\begin{rem}
Fix an $R$-module $M$, and consider the tensor-hom adjunction $(M\otimes_R-,\Hom_R(M,-))$.
Then the unit and counit are given by $X\mapsto\alpha_{X,M}$ and $Y\mapsto\beta_{M,Y}$.
The above lemma is actually a consequence of a general fact on the unit and counit of an adjunction.
\end{rem}

Let $R$ be a local ring with residue field $k$.
For an $R$-module $M$ we denote by $\mu(M)$ the the {\em minimal number of generators} of $M$, i.e., $\mu(M)=\dim_k(M\otimes_Rk)$.
The next proposition consists of dual statements.
The duality is more visible if we notice $N\ne0\Leftrightarrow k\otimes_RN\ne0$ and $\depth_RN=0\Leftrightarrow\Hom_R(k,N)\ne0$.
The first assertion of the proposition plays a key role in the proof of the main result of this section.

\begin{prop}\label{lhom van}
Let $R$ be local and $M,N$ be $R$-modules.
Then $M$ is free in each of the two cases below.\\
{\rm(1)} $N\ne0$ and $\Tor_1^R(\tr M,M\otimes_RN)=0$.\qquad
{\rm(2)} $\depth_RN=0$ and $\Ext_R^1(\tr M,\Hom_R(M,N))=0$.
\end{prop}
 
\begin{proof}
Let $\lambda:M\otimes M^\ast\to R$ be the homomorphism given by $\lambda(x\otimes f)=f(x)$ for $x\in M$ and $f\in M^\ast$.
Set $I=\image\lambda$.
There is an exact sequence $M\otimes M^\ast\xrightarrow{\lambda}R\to R/I\to0$.
We prove the proposition by induction on $\mu(M)$.
If $\mu(M)=0$, then $M=0$ and $M$ is free.
Let $\mu(M)>0$.
It suffices to show $I=R$.
Indeed, it means that $\lambda$ is surjective, which is equivalent to saying that $R$ is isomorphic to a direct summand of $M$; see \cite[Proposition 2.8(iii)]{L} for instance.
Then there is an $R$-module $M^\prime$ such that $M \cong R\oplus M^\prime$.
Note that $\mu(M^\prime)<\mu(M)$, and that $\Tor_1(\tr M^\prime ,M^\prime\otimes N)$ and $\Ext^1(\tr M',\Hom(M',N))$ are direct summands of $\Tor_1(\tr M,M\otimes N)$ and $\Ext^1(\tr M,\Hom(M,N))$, respectively.
The induction hypothesis implies that $M^\prime$ is free, and so is $M$.
Thus, from now on, we prove $I=R$ in each of the two cases in the proposition.

(1) Set $L = M\otimes N$.
As $M\ne0\ne N$, we have $L\ne0$.
Lemma \ref{6} shows the surjectivity of $\beta_{M,L}$.
As $\Tor_1(\tr M,L)=0$, the map $\varpi:M^\ast \otimes L\to\Hom(M,L)$ given by $\varpi(f\otimes z)(x)=f(x)z$ for $f\in M^\ast$, $z\in L$ and $x\in M$ is surjective by \cite[Lemmas (3.8) and (3.9)]{YO}.
We see that the composition
$$
M\otimes M^\ast\otimes L\xrightarrow{M\otimes\varpi}M\otimes\Hom(M,L)\xrightarrow{\beta_{M,L}}L
$$
coincides with the map $\lambda\otimes L:M\otimes M^\ast\otimes L\to R\otimes L\cong L$.
As $M\otimes\varpi$ and $\beta_{M,L}$ are surjective, so is $\lambda\otimes L$.
Hence $R/I\otimes_RL=0$ and $L=IL$.
Since $L\ne0$, Nakayama's lemma implies $I=R$.

(2) Set $L=\Hom(M,N)$.
As $M\ne0$ and $\depth N=0$, we have $L\ne0$ by \cite[Exercises 1.2.27]{BH}.
Lemma \ref{6} shows the injectivity of $\alpha_{L,M}$.
As $\Ext^1(\tr M,L)=0$, the map $\varpi:M\otimes L\to\Hom(M^\ast,L)$ given by $\varpi(x\otimes z)(f)=f(x)z$ for $f\in M^\ast$, $z\in L$ and $x\in M$ is injective by \cite[Proposition (2,6)(a)]{AB}.
The composition below coincides with the map $\Hom(\lambda,L):L\cong\Hom(R,L)\to\Hom(M\otimes M^\ast,L)$.
$$
L\xrightarrow{\alpha_{L,M}}\Hom(M,M\otimes L)\xrightarrow{\Hom(M,\varpi)}\Hom(M,\Hom(M^\ast,L))\cong\Hom(M\otimes M^\ast,L).
$$
As $\Hom(M,\varpi),\alpha_{L,M}$ are injective, so is $\Hom(\lambda,L)$.
Thus $0=\Hom(R/I,L)\cong\Hom(R/I,\Hom(M,N))$.
By \cite[Exercises 1.2.27]{BH} again, we get equalities $\emptyset=\ass\Hom(R/I,\Hom(M,N))=\V(I)\cap\supp M\cap\ass N$.
Since the maximal ideal $\m$ belongs to the set $\supp M\cap\ass N$, we obtain the equality $I=R$.
\end{proof}

\begin{rem}
In the proof of the above proposition, $I=\image\lambda$ is by definition the {\em trace ideal} of the $R$-module $M$.
Trace ideals have been investigated in relation to the Auslander--Reiten conjecture in several works of Lindo and her collaborators; see \cite{L2} for instance.
\end{rem}

Next we state a lemma which is frequently used later.

\begin{lem}\label{Ext and Tor}
Let $R$ be a local ring of dimension $d$.
Let $M$ and $N$ be $R$-modules.
\begin{enumerate}[\rm(1)]
\item
If $\pd_RM<\infty$ and $N\ne0$, then $\sup\{i\in\N\mid\Ext_R^i(M,N)\ne0\}=\pd_RM=\depth R-\depth_RM$.
\item
Assume that $R$ is Cohen--Macaulay and has a canonical module $\omega$.
Suppose that $M$ is locally free on the punctured spectrum of $R$ and that $N$ is maximal Cohen--Macaulay.
Then there is an isomorphism $\Ext_R^{d+i}(M,N^\dag)^\vee\cong\Tor_i^R(M,N)$ for every integer $i>0$.
\end{enumerate}
\end{lem}

\begin{proof}
(1) The assertion follows by \cite[Page 154, Lemma 1(iii)]{Mat} and the Auslander--Buchsbaum formula.

(2) Similarly as in the proof of Proposition \ref{5}, a spectral sequence $E_2^{p,q}=\Ext^p(\Tor_q(M,N),\omega)\Rightarrow H^{p+q}=\Ext^{p+q}(M,N^\dag)$ exists.
The $R$-module $\Tor_q(M,N)$ has finite length if $q>0$.
We have $E_2^{p,q}=0$ if $p>d$ or if $p<d$ and $q>0$.
Thus $H^{d+i}\cong E_2^{d,i}$ for $i\ge0$.
The assertion follows by \cite[Corollary 3.5.9]{BH}.
\end{proof}

For an integer $n\ge0$, we denote by $\x^n(R)$ the set of prime ideals of $R$ with height at most $n$.
For a property $\P$ of local rings (resp. of modules over a local ring) and a subset $X$ of $\spec R$, we say that $R$ (resp. an $R$-module $M$) {\em locally satisfies $\P$ on $X$} if the local ring $R_\p$ (resp. the $R_\p$-module $M_\p$) satisfies $\P$ for every $\p\in X$.
The following theorem is one of the main results of this paper.

\begin{thm}\label{them1}
Let $R$ be a Cohen--Macaulay local ring of dimension $d\ge2$, and let $1\le n\le d-1$ be an integer.
Let $M$ be an $R$-module which locally has finite projective dimension on $\x^n(R)$.
Let $N$ be a maximal Cohen--Macaulay $R$-module with $\supp N=\spec R$.
Suppose that $\Ext_R^i(M,N)=0$ for all $1\le i\le d$ and $\Ext_R^j(M^\ast,\Hom_R(M,N))=0$ for all $n\le j\le d-1$.
Then $M$ is a free $R$-module.
\end{thm}

\begin{proof}
We prove the theorem step by step.

(1) We claim that $M$ is locally free on $\x^n(R)$.
In fact, for every $\p\in\x^n(R)$ and every integer $1\le i\le n$ we have $\pd_{R_\p}M_\p<\infty$, $N_\p\ne0$ and $\Ext_{R_\p}^i(M_\p,N_\p)=0$.
The claim follows from Lemma \ref{Ext and Tor}(1).

(2) We consider the case where $R$ possesses a canonical module $\omega$.

(i) First we deal with the case $n=d-1$.
In this case, because of (1), the $R$-module $M$ is locally free on the punctured spectrum of $R$.
It follows from Corollary \ref{lem1} that the $R$-module $M\otimes N^\dag$ is maximal Cohen--Macaulay, and hence so is $(M\otimes N^\dag)^\dag\cong\Hom_R(M,N)$.
There are isomorphisms
$$
\begin{array}{l}
0=\Ext^{d-1}(M^\ast,\Hom(M,N))^\vee
\cong\Ext^{d+1}(\tr M,\Hom(M,N))^\vee\\
\phantom{0=\Ext^{d-1}(M^\ast,\Hom(M,N))^\vee}
\cong\Tor_1(\tr M,\Hom(M,N)^\dag)
\cong\Tor_1(\tr M,M\otimes N^\dag).
\end{array}
$$
Here, the first isomorphism follows from the isomorphism $M^\ast\cong\syz^2 \tr M$ up to free summands, the second from Lemma \ref{Ext and Tor}(2), and the third from the isomorphisms $\Hom(M,N)^\dag\cong(M\otimes N^\dag)^{\dag\dag}\cong M\otimes N^\dag$.
By virtue of Proposition \ref{lhom van}(1), we conclude that $M$ is a free $R$-module.

(ii) Next, we handle the case where $n$ is general.
Fix an integer $t\ge n$.
We prove by induction on $t$ that $M$ is locally free on $\x^t(R)$.
We know that this holds true for $t=n$.
Let $t>n$ and assume that $M$ is locally free on $\x^{t-1}(R)$.
Pick $\p\in\x^t(R)$.
What we want to show is that $M_\p$ is $R_\p$-free, and for this we may assume $\height\p=t$.
Then, $R_\p$ is a Cohen--Macaulay local ring of dimension $t\ge2$, and $(\omega_R)_\p$ is a canonical module of $R_\p$.
The $R_\p$-module $M_\p$ is locally free on $\x^{t-1}(R_\p)$.
The $R_\p$-module $N_\p$ is maximal Cohen--Macaulay and with $\supp_{R_\p}N_\p=\spec R_\p$.
We have $\Ext_{R_\p}^i(M_\p,N_\p)=0$ for every $1\le i\le t$ and $\Ext_{R_\p}^{t-1}(\Hom_{R_\p}(M_\p,R_\p),\Hom_{R_\p}(M_\p,N_\p))=0$.
Therefore, $M_\p$ is $R_\p$-free by (i).
We have shown that $M$ is locally free on $\x^t(R)$ for all $t\ge n$.
Hence $M$ is locally free on $\x^d(R)$, which means that $M$ is $R$-free.

(3) Finally, we address the general case.
Denote by $\widehat{(-)}$ the $\m$-adic completion.
Let $P\in\spec\widehat R$ and $\p=P\cap R\in\spec R$.
Then the rings $\widehat R$ and $\widehat{R}_P$ are faithfully flat over $R$ and $R_\p$, respectively.
Since $\height P=\height\p+\height P/\p\widehat R$ by \cite[Theorem 15.1]{Mat}, it holds that if $P\in\x^n(\widehat R)$, then $\p\in\x^n(R)$.
Each $R$-module $X$ satisfies $\widehat{X}_P\cong X_\p\otimes_{R_\p}\widehat{R}_P$.
Now it is observed that the conditions on $R,M,N$ are satisfied by $\widehat R,\widehat M,\widehat N$.
Since $\widehat R$ has a canonical module, we can apply (2) to see that $\widehat M$ is $\widehat R$-free.
Consequently, $M$ is $R$-free.
\end{proof}

\begin{ques}
As is seen, Proposition \ref{lhom van}(1) plays a crucial role in the proof of Theorem \ref{them1}.
Can we prove something similar (possibly a statement dual to the theorem) by using Proposition \ref{lhom van}(2)?
\end{ques}

The result below is a direct corollary of the above theorem.

\begin{cor}\label{11}
Let $R$ be a Cohen--Macaulay local ring of dimension $d\ge2$ with a canonical module $\omega$.
Let $1\le n\le d-1$.
Let $M$ be a maximal Cohen--Macaulay $R$-module which is locally of finite projective dimension on $\x^n(R)$.
If $\Ext_R^i(M^\ast,M^\dag)=0$ for every $n\le i\le d-1$, then $M$ is free.
\end{cor}

\begin{proof}
As $\Hom_R(\omega,\omega)\cong R$, we have $\supp\omega=\spec R$.
Since $M$ is maximal Cohen--Macaulay, we have $\Ext^{>0}(M,\omega)=0$.
Applying Theorem \ref{them1} to $N=\omega$, we see that the assertion holds.
\end{proof}

Let $m,n \in\N\cup\{\infty\}$.
By $\G_{m,n}$ we denote the full subcategory of $\mod R$ consisting of modules $M$ such that $\Ext^i_R(M,R)=0$ for all $1\le i\le m$ and $\Ext^j_R(\tr M,R)=0$ for all $1\le j\le n$.
Applying Theorem \ref{them1}, we obtain the following result, which is also one of the main results of this paper.

\begin{thm}\label{k1}
Let $R$ be a Cohen--Macaulay local ring of dimension $d\ge2$, and let $1\le n\le d-1$ be an integer.
Let $M$ be an $R$-module which locally has finite projective dimension on $\x^n(R)$ and which satisfies $\Ext^j_R(M,M)=0$ for all $n \le j \le d-1$.
Then $M$ is free in each of the following two cases.\\
{\rm(1)} $M\in\G_{2d-t+1,t}$ for some integer $0\le t\le d+2$.\qquad
{\rm(2)} $\Ext^i_R(M,R)=0$ for all integers $1 \le i \le 2d+1$.
\end{thm}

\begin{proof}
Note that (2) is the situation where we let $t=0$ in (1).
Thus we have only to show the theorem in case (1).
Since $\pd_{R_\p}M_\p<\infty$ and $\Ext_{R_\p}^i(M_\p,R_\p)=0$ for all $\p\in\x^n(R)$ and $1\le i\le n$, the $R$-module $M$ is locally free on $\x^n(R)$ by Lemma \ref{Ext and Tor}(1).
Put $L=\syz^{d-t+2}M$.
Since $M$ is locally free on $\x^n(R)$, so is the $R$-module $L^\ast$.
Since $M \in \G_{2d-t+1,t}$, we get $L\in \G_{d-1,d+2} \subseteq \G_{0,2}$ and $\syz^2\tr L\in \G_{d,d+1}\subseteq\G_{d,0}$ by \cite[Proposition 1.1.1]{I}, while $L^\ast\cong\syz^2\tr L$ up to free summands.
Thus $L$ is reflexive by \cite[Proposition (2.6)(a)]{AB}, and $\Ext^i(L^\ast,R)=0$ for all $1 \le i \le d$.
On the other hand, for each integer $n \le j \le d-1$ there are isomorphisms $0=\Ext^j(M,M)\cong\Ext^{j+1}(M,\syz M)\cong\cdots\cong\Ext^{j+d-t+2}(M,\syz^{d-t+2}M)$ as $\Ext^k(M,R)=0$ for all $j\le k\le j+d-t+2$.
Hence $\Ext^j((L^\ast)^\ast,(L^\ast)^\ast)\cong\Ext^j(L,L)\cong\Ext^{j+d-t+2}(M,L)=0$.
Theorem \ref{them1} implies that $L^\ast$ is free, and so is $L^{\ast\ast}\cong L=\syz^{d-t+2}M$.
Therefore, $M$ has finite projective dimension.
If $t=d+2$, then $M$ coincides with $L$, which is free.
If $t<d+2$, then $2d-t+1\ge d$ and $\Ext^i_R(M,R)=0$ for all $1 \le i \le d$, whence $M$ is free by Lemma \ref{Ext and Tor}(1).
The proof of the theorem is now completed.
\end{proof}

Let $n\ge0$ be an integer, and $M$ be an $R$-module.
We say that $M$ satisfies {\em Serre's condition $(\s_n)$} if $\depth_{R_\p}M_\p\ge\inf\{n,\height\p\}$ for all $\p\in\spec R$.
Below is a direct corollary of the two theorems given above.

\begin{cor}\label{10}
Let $R$ be a $d$-dimensional Cohen--Macaulay ring, and assume that $R$ is locally a complete intersection on $\x^1(R)$.
Then an $R$-module $M$ is projective under each of the following two conditions.
\begin{enumerate}[\rm(1)]
\item
$\Ext_R^i(M^\ast,R)=\Ext_R^j(M,M)=0$ for all $1\le i\le d$ and $1\le j\le\max\{2,d-1\}$, and $M$ satisfies $(\s_2)$.
\item
$\Ext_R^i(M,R)=\Ext_R^j(M,M)=0$ for all $1\le i\le 2d+1$ and $1\le j\le\max\{2,d-1\}$.
\end{enumerate}
\end{cor}

\begin{proof}
For each $\p\in\x^1(R)$, the local ring $R_\p$ is a complete intersection and $\Ext_{R_\p}^i(M_\p,M_\p)=0$ for $i=1,2$.
Applying \cite[Proposition 1.8]{ADS} to the $\p R_\p$-adic completion of $R_\p$ and using Lemma \ref{Ext and Tor}(1), we observe that $M_\p$ is $R_\p$-free.
Therefore, $M$ is locally free on $\x^1(R)$.

(1) Both $R$ and $M$ satisfy Serre's condition $(\s_2)$.
It follows from \cite[Theorem 3.6]{EG} that $M$ is reflexive.
We have $\Ext_R^i(M^\ast,R)=\Ext_R^j((M^\ast)^\ast,(M^\ast)^\ast)=0$ for all $1\le i\le d$ and $1\le j\le d-1$.
Applying Theorem \ref{them1} after localization at each prime ideal shows that $M^\ast$ is projective, and so is $(M^\ast)^\ast\cong M$.

(2) Fix $\p\in\spec R$.
If $\height\p\le1$, then $M_\p$ is $R_\p$-free as mentioned above.
If $\height\p\ge2$, then $2\le\dim R_\p\le d$, and applying Theorem \ref{k1}(2) to $R_\p$, $M_\p$ and $n=1$ shows $M_\p$ is $R_\p$-free.
Thus $M$ is $R$-projective.
\end{proof}

The following result is a corollary of Corollary \ref{10}(2).

\begin{cor}\label{13}
Let $S$ be a locally excellent ring of dimension $e$ which is locally a complete intersection on $\x^1(S)$.
Let $\xx=x_1,\dots,x_n$ be a sequence of elements of $S$ which is locally regular on $\V(\xx)$.
Let $R=S/(\xx)$ be a Cohen--Macaulay ring.
Let $M$ be an $R$-module.
Suppose that $\Ext_R^i(M,R)=\Ext_R^j(M,M)=0$ for all $1\le i\le 2e+1$ and $1\le j\le\max\{2,e-1\}$.
Then $M$ is a projective $R$-module.
\end{cor}

\begin{proof}
Passing through localization at each $\p\in\V(\xx)$, we may assume that $(S,\n)$ is an excellent local ring.
Let $\widehat S$ be the $\n$-adic completion of $S$.
Let $P\in\x^1(\widehat S)$ and put $\p=P\cap S$.
Similarly as in part (3) in the proof of Theorem \ref{them1}, we have $\p\in\x^1(S)$ and the induced map $S_\p\to\widehat S_P$ is faithfully flat.
By assumption, $S_\p$ is a complete intersection.
As $S$ is an excellent local ring, the closed fiber $\widehat S_P/\p\widehat S_P$ is regular.
It follows from \cite[Remark 2.3.5]{BH} that $\widehat S_P$ is a complete intersection.
Thus $\widehat S$ is locally a complete intersection on $\x^1(\widehat S)$.
Replacing $S$ with $\widehat S$, we may assume that $S$ is complete.

We prove the assertion by using induction on $n$.
When $n=0$, we have $R=S$ and are done by Corollary \ref{10}(2).
Let $n>0$ and set $T=S/(x_1,\dots,x_{n-1})$.
The element $x_n$ is $T$-regular and $R=T/x_nT$.
As $T$ is complete and $\Ext_R^2(M,M)=0$, we find a $T$-module $N$ such that $x_n$ is $N$-regular and $M\cong N/x_n N$ by \cite[Proposition 1.6]{ADS}.
There is an exact sequence $0\to N\xrightarrow{x_n}N\to M\to0$.
By \cite[Lemma 3.1.16]{BH} we get exact sequences $\Ext_T^i(N,T)\xrightarrow{x_n}\Ext_T^i(N,T)\to\Ext_R^i(M,R)=0$ and $\Ext_T^j(N,N)\xrightarrow{x_n}\Ext_T^j(N,N)\to\Ext_R^j(M,M)=0$ for $1\le i\le 2e+1$ and $1\le j\le\max\{2,e-1\}$.
By Nakayama's lemma, for such $i,j$ we have $\Ext_T^i(N,T)=\Ext_T^j(N,N)=0$.
By the induction hypothesis, $N$ is $T$-free, whence $M$ is $R$-free.
\end{proof}

Here we recall a celebrated long-standing conjecture due to Auslander and Reiten \cite{AR}.

\begin{conj}[Auslander--Reiten]\label{8}
Every $R$-module $M$ such that $\Ext^{>0}_R(M,M\oplus R)=0$ is projective.
\end{conj}

The corollary below is deduced from Theorem \ref{k1}(2), which gives positive answers to Conjecture \ref{8}.

\begin{cor}\label{k2}
\begin{enumerate}[\rm(1)]
\item
Let $R$ be a Cohen--Macaulay ring.
Let $M$ be an $R$-module with $\Ext^{>0}_R(M,M\oplus R)=0$.
If $M$ is locally of finite projective dimension on $\x^1(R)$, then $M$ is a projective $R$-module.
\item
Let $R$ be a Cohen--Macaulay ring.
Suppose that $R$ locally satisfies Conjecture \ref{8} on $\x^1(R)$.
Then $R$ locally satisfies Conjecture \ref{8} on $\spec R$.
In particular, $R$ satisfies Conjecture \ref{8}.
\item
Conjecture \ref{8} holds true for every Cohen--Macaulay normal ring $R$.
\end{enumerate}
\end{cor}

\begin{proof}
(1) We may assume that the ring $R$ is local.
Let $\m$ be the maximal ideal of $R$.
If $\dim R\le1$, then $\m$ belongs to $\x^1(R)$ and $M$ has finite projective dimension.
As $\Ext^{>0}(M,R)=0$, Lemma \ref{Ext and Tor}(1) implies that $M$ is free.
If $\dim R\ge2$, then we apply Theorem \ref{k1}(2) to $n=1$ to deduce the assertion.

(2) Fix $\p\in\spec R$.
Let $M\in\mod R_\p$ with $\Ext^{>0}_{R_\p}(M,M\oplus R_\p)=0$.
Write $M=N_\p$ with $N\in\mod R$.
For $Q\in\x^1(R_\p)$ and $\q=Q\cap R$,  we have $\p\supseteq\q\in\x^1(R)$ and $\Ext^{>0}_{R_\q}(N_\q,N_\q\oplus R_\q)=0$.
By assumption, $N_\q$ is $R_\q$-free.
Hence $M$ is locally free on $\x^1(R_\p)$, and $R_\p$-free by (1).
Thus $R_\p$ satisfies Conjecture \ref{8}.

(3) As the ring $R$ is normal, it is locally a discrete valuation ring or a field on $\x^1(R)$.
Hence Conjecture \ref{8} locally holds on $\x^1(R)$.
It follows from (2) that Conjecture \ref{8} holds true for $R$.
\end{proof}

Finally, we compare our results obtained in this section with ones in the literature.

\begin{rmk}\label{12}
\begin{enumerate}[(1)]
\item
Corollary \ref{11} implies \cite[Theorem 1.4]{secm} (a bit weaker version of \cite[Corollary 3.9(1)]{ST}).
Indeed, let $R,M$ be as in \cite[Theorem 1.4]{secm}.
Similarly as in the beginning of \cite[Proof of Theorem 1.4]{secm} $M$ is reflexive and locally free on $\x^n(R)$.
Then Corollary \ref{11} applies to $M^\ast$ to yield that $M$ is free.
\item
Theorem \ref{k1}(1) partly refines \cite[Theorem 3.14]{DEL}.
When $R$ is Cohen--Macaulay and $M\in\G_{0,t}$, the same as Theorem \ref{k1}(1) for $n=1$ is asserted by \cite[Theorem 3.14]{DEL} under the assumption that $R$ is a quotient of a regular local ring that is locally Gorenstein on $\x^1(R)$ and contains $\Q$, and $\depth M=t$.
\item
Let $R$ be a Cohen--Macaulay local ring of dimension $d>0$ with a canonical module $\omega$.
By Theorem \ref{k1}(1) and \cite[Theorem 3.6]{EG}, $R$ is Gorenstein if it is locally Gorenstein on $\x^1(R)$ and $\Ext^i_R(\omega,R)=0$ for $1\le i\le2d-1$.
This is a weak version of \cite[Theorem 2.1]{ABS}, which asserts a generically Gorenstein local ring $R$ of dimension $d$ with a dualizing complex $D$ is Gorenstein if $\Ext^i_R(D,R)=0$ for $1 \le i \le d$.
\item
Theorem \ref{k1}(2) highly refines \cite[Corollary 1.5]{ST} in the case of a Cohen--Macaulay ring.
More precisely, if $R$ is a Cohen--Macaulay ring, then \cite[Corollary 1.5]{ST} asserts the same as Theorem \ref{k1}(2) under the additional assumptions that $M$ is locally free on $\x^n(R)$, that $\Ext^{>0}_R(M,R)=0$, and that either $\Ext_R^i(\Hom_R(M,M),R)=0$ for all $n<i\le\dim R$ or $\Hom_R(M,M)$ has finite G-dimension.
\item
Corollary \ref{10}(1) considerably improves \cite[Theorem 1.3]{HL}\footnote{There is a typo in \cite[Main Theorem and Theorem 1.3]{HL}: it is necessary to assume $\Ext_R^2(M,M)=0$ when $d=1$ for the proof to work. It suffices to replace the range $[1,d]$ of integers $i$ such that $\Ext_R^i(M,M)=0$ with $[1,\max\{2,d\}]$.}.
To be precise, Corollary \ref{10}(1) removes the assumptions imposed in \cite[Theorem 1.3]{HL} that $R$ is local and complete, that $R$ is either Gorenstein or contains $\Q$, that $M$ has constant rank, and that $\Ext_R^i(M,R)=0$ for all $1\le i\le d$.
Moreover, Corollary \ref{10}(1) shrinks the range $[1,\max\{2,d\}]$ of integers $j$ with $\Ext_R^j(M,M)=0$ to $[1,\max\{2,d-1\}]$, and relaxes the assumption of $M$ being maximal Cohen--Macaulay to $M$ satisfying $(\s_2)$.
\item
Corollary \ref{10}(2) reaches the same conclusion as \cite[Theorem 3.16]{DEL} by assuming more vanishings of Ext modules and instead weakening normality to the local complete intersection property in codimension one and removing the assumption that $M$ is maximal Cohen--Macaulay and with $\Hom_R(M,M)$ free.
\item
Corollary \ref{13} refines \cite[Main Theorem]{HL}\setcounter{footnote}{0}\footnotemark, that is to say, Corollary \ref{13} removes the assumptions given in \cite[Main Theorem]{HL} that $S$ is either Gorenstein or contains $\Q$ and that $M$ has constant rank, and shrinks the range $[1,\max\{2,\dim S\}]$ of integers $j$ with $\Ext_R^j(M,M)=0$ to $[1,\max\{2,\dim S-1\}]$.
\item
Corollary \ref{k2}(2) extends \cite[Theorem 3]{Ar} from Gorenstein rings to Cohen--Macaulay rings.
\item
By virtue of \cite[Theorem 0.1]{HL}, Conjecture \ref{8} is known to hold true for a locally excellent Cohen--Macaulay normal ring $R$ that contains $\Q$.
Corollary \ref{k2}(3) removes from this statement the assumptions that $R$ is locally excellent and that $R$ contains $\Q$.
\end{enumerate}
\end{rmk}

\begin{ac}
The authors would like to thank Olgur Celikbas, Hiroki Matsui, Hidefumi Ohsugi, Kazuho Ozeki and Shunsuke Takagi for their valuable comments.
The authors are also indebted to the anonymous referee for pointing out an error in Corollary \ref{a} in the previous version of the paper, which is corrected in this version.
\end{ac}

\end{document}